\documentclass[aap]{imsart}

\RequirePackage{amsthm,amsmath,amsfonts,amssymb}
\RequirePackage[numbers,sort&compress]{natbib}
\RequirePackage[colorlinks,citecolor=blue,urlcolor=blue]{hyperref}
\RequirePackage{graphicx}

\startlocaldefs
\usepackage{notation}
\usepackage{todonotes}

\theoremstyle{plain}

\newtheorem{theorem}{Theorem}[section]
\newtheorem{lemma}[theorem]{Lemma}
\theoremstyle{remark}
\newtheorem{definition}[theorem]{Definition}

\newtheorem*{fact}{Fact}
\newtheorem{corollary}[theorem]{Corollary}
\newtheorem{remark}[theorem]{Remark}
\newtheorem{proposition}[theorem]{Proposition}

\endlocaldefs

\begin{document}

\begin{frontmatter}
\title{A New Bound on the Cumulant Generating Function of Dirichlet Processes}
\runtitle{Bound on the CGF of Dirichlet Processes}


\begin{aug}

\author[A]{\fnms{Pierre}~\snm{Perrault}\ead[label=e1]{pierre.perrault@outlook.com}},
\author[B]{\fnms{Denis}~\snm{Belomestny}\ead[label=e2]{denis.belomestny@uni-due.de}},
\author[C]{\fnms{Pierre}~\snm{M{\'e}nard}\ead[label=e3]{pmenard@meta.com}},
\author[D]{\fnms{{\'E}ric}~\snm{Moulines}\ead[label=e4]{eric.moulines@polytechnique.edu}},
\author[E]{\fnms{Alexey}~\snm{Naumov}\ead[label=e5]{anaumov@hse.ru}},
\author[D]{\fnms{Daniil}~\snm{Tiapkin}\ead[label=e6]{daniil.tiapkin@polytechnique.edu}},
\author[C]{\fnms{Michal}~\snm{Valko}\ead[label=e7]{mir@meta.com}}
\address[A]{IDEMIA\printead[presep={,\ }]{e1}}

\address[B]{Duisburg-Essen University, Germany\printead[presep={,\ }]{e2}}

\address[C]{Meta\printead[presep={,\ }]{e3,e7}}

\address[D]{Centre de Math{\'e}matiques Appliqu{\'e}es, CNRS, {\'E}cole Polytechnique, Institut Polytechnique de Paris, France\printead[presep={,\ }]{e4,e6}}

\address[E]{HSE University, Russia\printead[presep={,\ }]{e5}}


\end{aug}

\begin{abstract}

In this paper, we introduce a novel approach for bounding the cumulant generating function (CGF) of a Dirichlet process (DP) \(X \sim \DP(\alpha \nu_0)\), using superadditivity. In particular, our key technical contribution is the demonstration of the superadditivity of \(\alpha \mapsto \log \mathbb{E}_{X \sim \DP(\alpha \nu_0)}[\exp( \mathbb{E}_X[\alpha f])]\), where $\mathbb{E}_X[f] = \int f dX$. This result, combined with Fekete's lemma and Varadhan's integral lemma, converts the known asymptotic large deviation principle into a practical upper bound on the CGF \(
\log\EEs{X\sim \DP\pa{\alpha\nu_0}}{\exp\pa{\EEs{X}{f}}}
\) for any \(\alpha > 0\). The bound is given by the convex conjugate of the scaled reversed Kullback-Leibler
divergence $\alpha\KL{\nu_0}{\cdot}$.
This new bound provides particularly effective confidence regions for sums of independent DPs, making it applicable across various fields.
\end{abstract}



\end{frontmatter}
\section{Introduction}
\label{sec:introduction}
The Dirichlet Process (DP) is a fundamental stochastic process in which each realization is itself a probability distribution. Originally introduced by Ferguson in the early 1970s \cite{ferguson1973bayesian}, the DP has become a fundamental tool in the field of nonparametric Bayesian statistics 
\cite{congdon2014applied,ghosal2017fundamentals,mueller2018nonparametric}. 
The use of DPs in statistical modeling provides a significant benefit: it enables the model’s complexity to adjust according to the data rather than requiring a fixed structure a priori.
Conceptually, the DP can be considered as a generalization of the Dirichlet distribution into infinite-dimensional spaces. Just as the Dirichlet distribution acts as the conjugate prior for categorical distributions, the DP serves as the conjugate prior for nonparametric, discrete distributions over infinite spaces.
Before proceeding further, we introduce the necessary notation. 
We focus on a compact metric space \( \Omega \), equipped with its Borel \(\sigma\)-algebra \( \cB(\Omega) \). The set \( \cM(\Omega) \) (and respectively \( \cM_1(\Omega) \)) denotes the space of finite non-negative (and probability) measures on \( \Omega \), and \( \cB(\cM_1(\Omega)) \) represents the Borel \(\sigma\)-algebra generated by the weak topology on \( \cM_1(\Omega) \). The set of continuous functions $f \colon \Omega \to \R$ is denoted by $\cC(\Omega)$.
For a given probability distribution \( p \), \( \E_p \) and \( \P_p \) denote the expectation and probability, respectively, with respect to \( p \). Alternatively, we use \( \E_{\xi \sim p} \) and \( \P_{\xi \sim p} \) to explicitly indicate that the random variable \( \xi \) follows the distribution \( p \).
\par
Consider a DP on $\Omega$, characterised by a scale parameter $\alpha>0$ and a base distribution $\nu_0\in \cM_1(\Omega)$ whose measure's law is denoted as $\DP\pa{\alpha\nu_0}$ and
 whose realization $X$ is a random probability measure on the space $\Omega$.
The original definition of the DP, introduced by Ferguson \cite{ferguson1973bayesian}, says that for any finite measurable partition $B_1\sqcup\dots\sqcup B_k=\Omega$, $(X(B_1),\dots,X(B_k))\sim \Dir(\alpha\nu_0(B_1),\dots,\alpha\nu_0(B_k))$, where 
$\Dir(\alpha_1,\ldots,\alpha_k)$
denotes the Dirichlet distribution with parameters $\alpha_1,\ldots,\alpha_k$. An alternative definition, known as the "stick-breaking" construction,
is 
$X=\sum_{k=1}^\infty \beta_k\prod_{\ell<k}(1-\beta_{\ell}) \delta_{\omega_k}$, where $\pa{\omega_k, \beta_k}\overset{iid}{\sim} \nu_0\times \text{Beta}\pa{1,\alpha}$ \cite{paisley2010simple}, where $\text{Beta}$ denotes the beta distribution.
Another important representation of the DP is analogous to the characterization of the Gamma distribution by \cite{lukacs1955characterization} and expressed as $X = G/G\pa{\Omega}$, where $G\sim\cG(\alpha\nu_0)$ is the standard Gamma process\footnote{One of the formulations of the Gamma process is derived from a Poisson process over the space $\Omega\times [0,\infty)$ with mean measure $\mu(d\omega,d t)=\alpha \nu_0(d\omega) t^{-1}e^{-t}dt$. Drawing a sample from this Poisson process generates an infinite set of atoms $(\omega_k,t_k)_{k\geq 1}$. The Gamma process can then be constructed as $G\triangleq \sum_{k=1}^\infty t_k \delta_{\omega_k}.$} on $\Omega$ with shape parameter $\alpha\nu_0$, i.e., $G(A)\sim \text{Gamma}\pa{\alpha\nu_0(A),1}$ for any measurable set $A\subset \Omega$ \cite{casella2005christian}. 

Given the widespread use of Dirichlet Processes, understanding concentration phenomena \cite{bernstein1924modification,hoeffding1963probability} and large deviation principles (LDPs, see Definition~\ref{def:LDP}) \cite{cramer1937nouveau,Varadhan1966asymptotic,bahadur1979large} is crucial, particularly in the context of their applications to fields such as machine learning \cite{blei2005variational}, reinforcement learning \cite{osband2013more,osband2017why}, topic modeling \cite{blei2003latent,teh2006hierarchical}, among others. 
In this paper,
we will focus on the study of concentration bounds for DPs. Indeed, the literature on LDPs for DPs is already well-established \cite{doss1982tails, lynch1987large,ganesh2000large,feng2007large,feng2023hierarchical} and it is known that the probability that $X\sim \DP\pa{\alpha\nu_0}$ deviates from another distribution $\nu$ decreases exponentially (with respect to $\alpha$) at a rate given by the Kullback–Leibler (KL) divergence (aka relative entropy) of $\nu_0$ with respect to $\nu$, defined as
\[\KL{\nu_0}{\nu}\triangleq\left\{
    \begin{array}{ll}
        \EEs{\nu_0}{ \log\pa{\frac{d\nu_0}{d\nu}}} & \mbox{if } \nu_0 \ll \nu \\
        \infty & \mbox{otherwise.}
    \end{array}
\right.\]
On the non-asymptotic side, research related to DPs remains sparse. The existing studies utilize the Gamma process-based representation defined above and afterward use a closed-form of the moment-generating function (MGF) formula for the Gamma process. This approach has been explored in various works; see \cite{van1997gamma, vershik2001remarks}.
In particular, for any $f\in \cC\pa{\Omega}$, such that \( f \leq 1 \) and \(\nu_0(f^{-1}(\{1\})) = 0\), the MGF of the Gamma process can be expressed as:
\begin{align}
\textstyle M_{\mathcal{G}(\alpha \nu_0)}(f) \triangleq \mathbb{E}_{G \sim \mathcal{G}(\alpha \nu_0)} \left[ \exp \left( \int f \, dG \right) \right] = \exp \left(-\alpha \mathbb{E}_{\nu_0} \left[ \log (1 - f) \right] \right). 
\label{rel:MGFGP}
\end{align}
While the Gamma process representation provides powerful tools for analyzing the concentration of an individual DP, extending this technique to multiple independent DPs is not feasible.


\par
Our primary objective in this paper is to establish new concentration bounds for multiple independent DPs. Moving forward, we intend to focus specifically on the MGF of the DP, as opposed to that of the Gamma process. 
Indeed, the MGF and its logarithm, the cumulant-generating function (CGF), are essential in both asymptotic and non-asymptotic statistical analyses. Specifically, they play a critical role in establishing LDPs using the Gärtner-Ellis theorem (see \cite{gartner1977large,ellis1984large}). Within a non-asymptotic realm, the multiplicative nature of MGFs for sums of independent random variables can be utilized to derive concentration inequalities through Chernoff bounds \cite{buldygin1980sub,buldygin2000metric,raginsky2013concentration,boucheron2013concentration,pisier2016subgaussian}.
In our context, bounding the CGF of a DP by a manageable expression would thus enable the derivation of new concentration inequalities for sums of independent DPs, which could have practical significance across various application domains.
\par
The study of moment-generating function (MGF) bounds, particularly sub-Gaussian properties, has been extensively explored due to its broad applicability across various fields. This line of research focuses on identifying the optimal proxy variance for different types of random variables, including discrete distributions such as the Bernoulli distribution \cite{buldygin2013sub,kearns2013large}, and continuous distributions like the Beta and Dirichlet distributions \cite{marchal2017sub}. While sub-Gaussian bounds are widely used due to their general applicability, they often do not provide as tight an estimate as those derived from the KL divergence (which are asymptotically optimal, as they match the LDP). Therefore, in this paper, we target a KL-based CGF bound on the DP.
\par
The rest of the article is structured as follows. The following section reviews LDP results for DPs and their relevance to concentration bounds. Next is our main result (Theorem~\ref{th:mgfbounddir}), which transforms a limit on a sequence of DP CGFs (focusing on the sequence considered in Varadhan's integral lemma \cite{Varadhan1966asymptotic}) into a bound.
Specifically, we utilize the superadditivity of the sequence (which is proved in the subsequent section), employing Fekete's superadditive lemma \cite{fekete1923verteilung}. The article concludes by applying this result to the stochastic semi-armed bandit problem \cite{kveton2015tight}.


\section{Large deviation principle}
We start by recalling the concept of a large deviation principle (LDP).
\begin{definition}[Rate function]
A function $I$ is a rate function if it is lower semicontinuous with values in $[0,\infty]$ (such that all level sets $\sset{x: I(x)\leq t},$ for $t\in [0,\infty)$, are closed). A rate function is good if the level sets are compact.  
The effective domain of $I$ is $D_I\triangleq \sset{x: I(x)<\infty}.$
\end{definition}
Notice, in the previous definition, we did not specify on which domain the rate function $I$ is defined. In our context, $I$ will be defined on $\cM_1\pa{\Omega}$.

\begin{definition}[Large deviation principle: LDP]\label{def:LDP}
    A sequence of probability measures $(\mu_n)$ satisfies an LDP with speed $n$ (we can avoid explicitly stating the speed if it is clear from the context) and rate function $I$ if:
    \begin{align*}
        (i)~~~~&\mbox{For all closed sets $F$, $\limsup_{n}\frac{1}{n}\log\mu_n\pa{F} \leq -\inf_{x\in F} I(x)$,}\\
        (ii)~~~~&\mbox{For all open sets $G$,
        $\liminf_{n}\frac{1}{n}\log\mu_n\pa{G} \geq -\inf_{x\in G} I(x)$.}
    \end{align*}
\end{definition}
The standard, and likely the most well-known, LDP result for DP is as follows.
\begin{theorem}[see \cite{ganesh2000large}]\label{thm:ldp}
$\pa{\DP\pa{\alpha\nu_0}}_\alpha$ satisfies a LDP with speed $\alpha$ and rate function $I(\nu) = \KL{\nu_0}{\nu},$ i.e.,
for all $B\in \cB\pa{\cM_1(\Omega)}$, if $B^{\mathrm{o}}$ (resp. $\bar B$) denotes the interior of $B$ (resp. the closure),
\begin{align*}-\inf_{\nu\in B^{\mathrm{o}}} I(\nu)\leq \liminf_{\alpha}\frac{1}{\alpha}\log\PPs{\DP\pa{\alpha\nu_0}}{B}\leq \limsup_{\alpha}\frac{1}{\alpha}\log\PPs{\DP\pa{\alpha\nu_0}}{B} \leq -\inf_{\nu\in \bar B} I(\nu).\end{align*}
\end{theorem}
\noindent
The rate function in Theorem~\ref{thm:ldp} is given by the reverse KL divergence $I(\cdot)=\KL{\nu_0}{\cdot}$, which is a dual to the rate function in the Sanov theorem \cite{sanov1958probability}.
 In \cite{ganesh1999inverse}, the reverse relation of the rate functions is explained as follows: ``in Sanov’s theorem we ask how likely the empirical
distribution is to be close to $\nu_0$, given that the true distribution is
$\nu$; whereas in the
Bayesian context we ask how likely it is that the true distribution is close to
$\nu_0$,
given that the empirical distribution is close to
$\nu$.'' There are several ways to get Theorem~\ref{thm:ldp}. For instance, \cite{lynch1987large} utilize an LDP on the Gamma process, linking to DPs through the characterization mentioned above, while \cite{ganesh2000large} rely on Varadhan's integral lemma (see Fact~\ref{fact:varadhan}). 


For some probability distribution $\nu\in \cM_1(\Omega)$, one of the central information-theoretic measures we are considering is defined as an infimum of Kullback-Leibler divergences:
for some real-valued continuous function $f\in \cC(\Omega)$ and some $u\in\R$, we define
\[\Kinf(\nu, u, f) \triangleq \inf_{\mu\in \cM_1(\Omega),~\EEs{\mu}{f}\geq u} \KL{\nu}{\mu},\]
where by convention, the infimum of the empty set equals to $\infty$.
Notice, $\Kinf$ is used to express the bounds of Theorem~\ref{thm:ldp} when $B$ is some deviation event.
This quantity can be interpreted as a distance from the measure $\nu$ to the set of all measures $\mu\in \cM_1(\Omega),~\EEs{\mu}{f}\geq u$, where the distance is measured by the KL-divergence.
The measure $\mu$ solving this optimization problem is called moment projection ($M$-projection) or reversed information projection ($rI$-projection), see \cite{csiszar2003information,bishop2006pattern,murphy2022probabilistic}. 
This is different from the more common information projection ($I$-projection), 
\(
\inf_{\mu\in \cM_1(\Omega),~\EEs{\mu}{f}\geq u} \KL{\mu}{\nu},
\)
 appearing, for example, in Sanov-type deviation bounds \cite{sanov1961probability}.
The $I$-projections have a geometric interpretation because the KL can be viewed as a Bregman divergence. The $M$-projections are not Bregman divergences and lack geometric interpretation. However, they are deeply connected to the maximum likelihood estimation when the measure $\nu$ is the empirical measure of a sample \cite[Lemma 3.1]{csiszar2004information}. Additionally, as we saw in Theorem~\ref{thm:ldp}, $M$-projections naturally appear as a rate function for a LDP in a Bayesian framework \cite{ganesh1999inverse}. They also naturally appear in lower (and sometimes upper) bounds for multi-armed bandits\footnote{We apply our results to this domain in Section~\ref{sec:cts}.} \cite{Lai1985asymptotically,burnetas1996optimal}.
Like the KL divergence, $\Kinf(\nu, u, f)$ admits the following variational formula.
\begin{lemma}[Variational formula for $\Kinf$ \cite{honda2015nonasymptotic,garivier2022klucbswitch}]
    For all $\nu\in \cM_1(\Omega)$, $f\in \cC(\Omega)$, $u\in [f_{\min},f_{\max})$, where $f_{\max}\triangleq \max_{x\in \Omega} f(x)$, $f_{\min}\triangleq \min_{x\in \Omega} f(x)$, we have
    \[\Kinf(\nu, u, f) = \max_{\lambda\in [0,1/(f_{\max}-u)]} \EEs{\nu}{\log\pa{1-\lambda \pa{f-u}}}.\]
    Moreover, if $\lambda^*$ is the value at which the above maximum is reached, then
 \[\EEs{\nu}{1/\pa{1-\lambda^*(f-u)}}\leq 1.\]
 In particular, $\nu\pa{f^{-1}\pa{\sset{f_{\max}}}}=0$ in the case $\lambda^*=1/(f_{\max}-u)$.
 \label{lem:Variational formula for Kinf}
\end{lemma}
\noindent
This formula is essential for deriving the deviation and concentration results involving~$\Kinf$: as a simple example, with the same assumptions and notations as in Lemma~\ref{lem:Variational formula for Kinf}, one can consecutively use Chernoff bound, Equation~\ref{rel:MGFGP} and Lemma~\ref{lem:Variational formula for Kinf} to get
\begin{align}\nonumber\textstyle
    \PPs{X\sim \DP\pa{\alpha \nu_0}}{\EEs{X}{f} \geq u} 
    &\leq\textstyle M_{\cG(\alpha\nu_0)}(\lambda^*\pa{f - u})
    \\&\textstyle= e^{-\alpha \EEs{\nu_0}{\log\pa{1-\lambda^*\pa{f - u}}}} = e^{-\alpha \Kinf(\nu_0, u, f)}.\label{rel:chernoff_(1)_lemma1}
\end{align}


In fact, LDP rate functions, in general, are often used in non-asymptotic concentration inequalities. For example, suppose we have a process consisting of real-valued i.i.d. random variables $(Y_i)$. 
Then, it is known that the sequence $\log\PP{{\sum_{i=1}^m Y_i}/m \geq x}$ is superadditive w.r.t.~$m\in \N^*$ (we recall that a function $f:\R\to\R$ is called superadditive on $A\subset \R$ if $f(x+y)\geq f(x)+f(y)$ for all $x,y\in A$), so from the superadditive lemma due to Fekete \cite{fekete1923verteilung},  for all $n\geq 1$,
\begin{align*}\textstyle\frac{1}{n}\log\PP{\frac{1}{n}{\sum_{i=1}^n Y_i} \geq x} &\leq \textstyle{\sup_{m\geq 1}\frac{1}{m}\log\PP{\frac{1}{m}{\sum_{i=1}^m Y_i} \geq x}} \\&= \textstyle {\lim_{m\to\infty}\frac{1}{m}\log\PP{\frac{1}{m}{\sum_{i=1}^m Y_i} \geq x}},\end{align*}
where the last quantity is the corresponding LDP rate function. 
Our aim in this paper is to use a similar superadditivity approach, but to bound the DP CGF instead.

\section{A bound through superadditivity of the CGF}

In this section, we explore how superadditivity for the DP CGF can translate a limit on the CGF into a bound. Specifically, our focus lies on the limit given by Varadhan's integral lemma \cite{Varadhan1966asymptotic}. A frequently encountered formulation of this lemma is the following.

\begin{fact}[Varadhan’s integral lemma \cite{varadhan1984large}]
\label{fact:varadhan}
Let $\cX$ be a complete separable metric space. Let
$(P_n)\in \cM_1\pa{\cX}^\N$ satisfying an LDP with a rate function $I(\cdot)$ and let $\phi\in \cC(\cX)$. 
Then
\[\lim_{n\to\infty}\frac{1}{n}\log\EEs{P_n}{\exp\pa{n\phi}} = \sup_{x\in \cX}\pa{\phi(x) - I(x)}.\]
\end{fact}

We can use this fact with the continuous mapping $\phi \colon \nu\in\cX\triangleq\cM_1\pa{\Omega} \mapsto \EEs{\nu}{f} $, where~$f\in\cC(\Omega)$. Indeed, since $\Omega$ is compact, $\cX=\cM_1\pa{\Omega}$ is
compact in the weak topology; additionally, it is separable and metrizable, for example, by the Lévy-Prokhorov metric \cite{prokhorov1956convergence}. We thus get the following result for DPs, using Theorem~\ref{thm:ldp} with $I(\cdot)=\KL{\nu_0}{\cdot}$ to get the needed LDP result.
\begin{corollary}\label{cor:varadhanDP}
    Let $f\in\cC(\Omega)$. Then, 
    \[
    \lim_{\alpha\to\infty} \frac{1}{\alpha}\log 
    \EEs{X\sim\DP\pa{\alpha\nu_0}}{\exp\pa{ \EEs{X}{\alpha f}}}=    \sup_{\nu\in \cM_1(\Omega)}\pa{\EEs{\nu}{f} - \KL{\nu_0}{\nu}}.\]
\end{corollary}

Our next step involves demonstrating the following Lemma~\ref{lem:superadditivedir}, with the aim of converting the preceding limit into an upper bound. 

\begin{lemma}[Superadditivity for the DP cumulant-generating function]\label{lem:superadditivedir}
Let $f\in \cC\pa{\Omega}$. Then, the function $\alpha\mapsto\log\EEs{X\sim \DP\pa{\alpha\nu_0}}{\exp\pa{\EEs{X}{\alpha f}}}$ is superadditive on $(0,\infty)$.
\end{lemma}
\noindent
The proof of Lemma~\ref{lem:superadditivedir} is postponed to section~\ref{sec:main_proof}. This in turn leads to the following Theorem~\ref{th:mgfbounddir}. 
\begin{theorem}
[Bound on the cumulant-generating function] Let $f\in \cC(\Omega)$.
Then,
$$ \log M_{\DP\pa{\alpha\nu_0}}(f) = \log\EEs{X\sim \DP\pa{\alpha\nu_0}}{\exp\pa{\EEs{X}{f}}}\leq 
\sup_{\nu\in \cM_1\pa{\Omega}}\pa{\EEs{\nu}{f} - \alpha\KL{\nu_0}{\nu}}.$$
\label{th:mgfbounddir}
\end{theorem}
\begin{proof}
The proof is a simple use of the Fekete's lemma \cite{fekete1923verteilung} coupled with Corollary~\ref{cor:varadhanDP}. More precisely,  Lemma~\ref{lem:superadditivedir} and the Fekete's lemma give
\[\lim_{\alpha\to\infty}\frac{1}{\alpha}\log\EEs{X\sim \DP\pa{\alpha\nu_0}}{\exp\pa{\EEs{X}{\alpha f}}} = \sup_{\alpha>0}\frac{1}{\alpha}\log\EEs{X\sim \DP\pa{\alpha\nu_0}}{\exp\pa{\EEs{X}{\alpha f}}},\]
so from Corollary~\ref{cor:varadhanDP}, for any $\alpha>0$,
\[\sup_{\nu\in \cM_1(\Omega)}\pa{\EEs{\nu}{f} - \KL{\nu_0}{\nu}} \geq \frac{1}{\alpha}\log\EEs{X\sim \DP\pa{\alpha\nu_0}}{\exp\pa{\EEs{X}{\alpha f}}}.\]
Multiplying this inequality by $\alpha$ and dividing $f$ by $\alpha$, we get the desired result.
\end{proof}
\begin{remark}
To the best of our knowledge, this bound is  new even for the special case of $X\sim\mathrm{Beta}(a,b)$. In this case, we have for $\lambda \geq 0$,  
\begin{equation}\label{eq:cgf_beta_bound}
    \psi(\lambda) \triangleq \log\EE{e^{\lambda (X-\EE{X})}}\leq \max_{s\in[a/(a+b),1]}\left( \lambda \pa{s-\frac{a}{a+b}} - (a+b)\kl{\frac{a}{a+b}}{s}\right)\,,
\end{equation}
where $\kl{p}{q}\triangleq p\log(\frac{p}{q})+(1-p)\log(\frac{1-p}{1-q}).$ The maximizer in \eqref{eq:cgf_beta_bound} is given by $s= \frac{\lambda - (a+b)+\sqrt{{\pa{\lambda - (a+b)}^2 + 4\lambda a}}}{2\lambda}$. 
Order-reversing property of the convex conjugate $\psi^*$ of $\psi$ implies that  
\[
\psi^*(\varepsilon) \geq (a+b)\kl{\frac{a}{a+b}}{\frac{a}{a+b} + \varepsilon}\,,
\]
and the Cramer method allows us to reproduce the proof for the tail probability bounds for the Beta distribution in terms of KL-divergence from \cite{dumbgen1998new} (see Remark~\ref{rk:sion_} for the generalization to DPs). However, let us stress that  our result is more general since it controls CGF, thus yielding similar inequalities for sums of independent Beta random variables (see Corollary~\ref{cor:indepDPs} for the generalization to DPs), where the Gamma distribution-based techniques of \cite{dumbgen1998new} becomes inapplicable.

\end{remark}
\begin{remark}
    In Theorem~\ref{th:mgfbounddir}, the bound is the convex conjugate of some reversed KL-divergence (as in Theorem~\ref{thm:ldp}). Since the KL is not symmetric, this is different from the KL used in the Donsker and Varadhan's formula \cite{donsker1983asymptotic} or in the original Sanov theorem \cite{sanov1958probability}.
\end{remark}

\begin{remark}
\label{rk:sion_}
 Theorem~\ref{th:mgfbounddir} can be used to recover \eqref{rel:chernoff_(1)_lemma1}, using Sion's theorem ($\cM_1\pa{\Omega}$ is a compact convex subset of the topological vector space of finite signed measures on $\Omega$): 
\begin{align*}\nonumber\textstyle
    \PPs{X\sim \DP\pa{\alpha \nu_0}}{\EEs{X}{f} \geq u} 
    &\leq\textstyle \inf_{\lambda \geq 0}M_{\DP\pa{\alpha \nu_0}}\pa{\lambda(f-u)}
    \\&\textstyle \leq 
    \exp\pa{ \inf_{\lambda \geq 0}\sup_{\nu\in \cM_1\pa{\Omega}}\pa{\lambda \EEs{\nu}{f-u} - \alpha\KL{\nu_0}{\nu}}}
    \\&\textstyle =
    \exp\pa{\sup_{\nu\in \cM_1\pa{\Omega}}\inf_{\lambda \geq 0}\pa{\lambda \EEs{\nu}{f-u} - \alpha\KL{\nu_0}{\nu}}}
    \\&\textstyle = e^{-\alpha \Kinf(\nu_0, u, f)}\,,
\end{align*}
letting $\lambda$ go to $0$ (resp. $\infty$) when $\EEs{\nu}{f-u}\geq 0$ (resp. $\EEs{\nu}{f-u}<0$) in the last equality.
\label{rk:sion_for_DP}
\end{remark}
We can go beyond Remark~\ref{rk:sion_for_DP} with the following result for the sum of independent DPs. Notably, the direct use of the multiplicative behavior of MGF allows us to bypass the need for the representation of DPs via Gamma processes.
\begin{corollary}[Confidence region for independent DPs]\label{cor:indepDPs}Consider $r\in \N^*$, $f_1,\dots, f_r\in \cC\pa{\Omega}$ and $\alpha_1\nu_1,\dots,\alpha_r\nu_r \in \cM\pa{\Omega}$.
    For $\delta \in (0,1)$, let $$\textstyle M_{\delta}\triangleq\sset{(\mu_j)_{j\in [r]}\in \pa{\cM_1\pa{\Omega}}^r,~\sum_{j=1}^r \alpha_j \KL{\nu_j}{\mu_j}\leq \log(1/\delta)}.$$
    Then,
    \[\textstyle\PPs{(X_j)\sim\otimes_{j\in [r]}\DP\pa{\alpha_j \nu_j}
    }{\sum_{j=1}^r
    \EEs{X_j}{f_j} > \sup_{(\mu_j)\in M_{\delta}} \sum_{j=1}^r\EEs{\mu_j}{f_j}}\leq \delta.\]
In addition, if $u\in \R$, then,
\[\PPs{(X_j)\sim\otimes_{j\in [r]}\DP\pa{\alpha_j \nu_j}
    }{\sum_{j=1}^r
    \EEs{X_j}{f_j}\geq u} \leq
\exp\pa{-\inf_{\substack{{(u_j)\in \R^r,} \\ {\sum_{j\in [r]} u_j = u}}}
 \sum_{j=1}^r \alpha_j\Kinf(\nu_j,u_j,f_j)}.\]
\end{corollary}


\begin{remark} The bound obtained in the second inequality of Corollary~\ref{cor:indepDPs} matches the LDP rate function obtained using the joint LDP \cite{chaganty1997large} and the contraction principle \cite{dembo2009large}. However, our bound is non-asymptotic.
\end{remark}

\begin{proof}[Proof of Corollary~\ref{cor:indepDPs}]
    Let's start by proving the first inequality. From strong duality via Slater condition, we have
$$
\sup_{\pa{\mu_j}\in M_{\delta}} \sum_{j=1}^r\EEs{\mu_j}{f_j}
=
\min_{\lambda\geq 0}
g(\lambda) = g(\lambda^*),
$$
where
$$ g(\lambda)\triangleq
\sup_{\pa{\mu_j}\in \cM_1\pa{\Omega}^r}{ \lambda \pa{\log(1/\delta) - \sum_{j=1}^r \alpha_j\KL{\nu_j}{\mu_j}}}
+
\sum_{j=1}^r\EEs{\mu_j}{f_j}. 
$$
If $\lambda^* = 0$, then the result is trivial as 
$$\PPs{(X_j)\sim\otimes_{j\in [r]}\DP\pa{\alpha_j \nu_j}
    }{\sum_{j=1}^r
    \EEs{X_j}{f_j} > \sup_{(\mu_j)\in M_{\delta}} \sum_{j=1}^r\EEs{\mu_j}{f_j}} = 0.$$
Thus, We consider the case $\lambda^*>0$.
From Chernoff inequality and Theorem~\ref{th:mgfbounddir}, 
\begin{align*}
&\PPs{(X_j)\sim\otimes_{j\in [r]}\DP\pa{\alpha_j \nu_j}
    }{\sum_{j=1}^r
    \EEs{X_j}{f_j} > \sup_{(\mu_j)\in M_{\delta}} \sum_{j=1}^r\EEs{\mu_j}{f_j}} \\&\leq \EEs{(X_j)\sim\otimes_{j\in [r]}\DP\pa{\alpha_j \nu_j}}{\exp\pa{\frac{1}{\lambda^*} \sum_{j=1}^r
    \EEs{X_j}{f_j}}}\exp\pa{-\frac{1}{\lambda^*} \sup_{(\mu_j)\in M_{\delta}} \sum_{j=1}^r\EEs{\mu_j}{f_j}}
\\&\leq
\exp\pa{\sup_{\pa{\mu_j}\in \cM_1\pa{\Omega}^r}\sum_{j=1}^r \pa{\frac{1}{\lambda^*}\EEs{\mu_j}{f_j} - \alpha_j\KL{\nu_j}{\mu_j}} - \frac{1}{\lambda^*}\sup_{(\mu_j)\in M_{\delta}} \sum_{j=1}^r\EEs{\mu_j}{f_j}}
\\&=
\delta 
\cdot
\exp\pa{\frac{1}{\lambda^*}g\pa{\lambda^*} - \frac{1}{\lambda^*}\sup_{(\mu_j)\in M_{\delta}} \sum_{j=1}^r\EEs{\mu_j}{f_j}} = \delta.
\end{align*}
For the second inequality, we have using Sion's minimax theorem
\begin{align*}
&\PPs{(X_j)\sim\otimes_{j\in [r]}\DP\pa{\alpha_j \nu_j}
    }{\sum_{j=1}^r
    \EEs{X_j}{f_j}\geq u} \\&\leq \exp\pa{\inf_{\lambda \geq 0}\sup_{\pa{\mu_j}\in \cM_1\pa{\Omega}^r}\lambda\pa{-u+\sum_{j=1}^r {\EEs{\mu_j}{f_j}}} - \sum_{j=1}^r\alpha_j\KL{\nu_j}{\mu_j}}
\\&= \exp\pa{\sup_{\pa{\mu_j}\in \cM_1\pa{\Omega}^r}\inf_{\lambda \geq 0}\lambda\pa{-u+\sum_{j=1}^r {\EEs{\mu_j}{f_j}}} - \sum_{j=1}^r\alpha_j\KL{\nu_j}{\mu_j}}
\\&=
\exp\pa{-\inf_{\substack{{(u_j)\in \R^r,} \\ {\sum_{j\in [r]} u_j = u}}}
 \sum_{j=1}^r \alpha_j\Kinf(\nu_j,u_j,f_j)}.
\end{align*}
\end{proof}
\section{Superadditivity for CGF of DPs}
\label{sec:main_proof}

\begin{proof}[Proof of Lemma~\ref{lem:superadditivedir}]

Without loss of generality, we can assume $0 \leq f \leq 1$ since $f\in \cC\pa{\Omega}$ is bounded as $\Omega$ is compact.

By the series decomposition of the exponential function, we observe that the desired inequality for $\alpha,\beta > 0$ can be obtained from the following moment inequality: for any integer $k \geq 0$,
\[\EEs{\pa{X,X'}\sim\DP\pa{\alpha \nu_0}\otimes \DP\pa{\beta \nu_0}}
{\pa{\EEs{X}{\alpha f}+\EEs{X'}{\beta f}}^k}
\leq \pa{\alpha+\beta}^k\EEs{X\sim\DP\pa{(\alpha+\beta) \nu_0}}{\EEs{X}{ f}^k}.
\]
This moment inequality can be reformulated as: for all $k\in \N,$ $Q_k\pa{f,\dots,f}\leq R_k\pa{f,\dots,f}$, where we define the following symmetric polynomials for $f_1,\dots,f_k$ measurable on $\Omega$:
  \begin{align*}
\textstyle Q_k\pa{f_1,\dots,f_k}&\textstyle\triangleq\EEs{\pa{X,X'}\sim\DP\pa{\alpha \nu_0}\otimes \DP\pa{\beta \nu_0}}{\prod_{\ell\in [k]}\pa{ \alpha 
\EEs{X}{f_{\ell}}
+ \beta 
\EEs{X'}{f_{\ell}}
}},
\\\textstyle R_k\pa{f_1,\dots,f_k}&\textstyle\triangleq (\alpha+\beta)^k\EEs{X\sim\DP\pa{(\alpha+\beta) \nu_0}}{\prod_{\ell\in [k]} 
\EEs{X}{f_{\ell}}
}.\end{align*}
Now, let $\pa{U_{\ell}}_{\ell\in [k]}\overset{iid}{\sim}\cU\pa{[0,1]}$ and define the indicator functions $F_{\ell}(\omega)\triangleq\II{U_{\ell}<f(\omega)}$. By Fubini's theorem, the multi-linearity of these symmetric polynomials, and using the independence of the uniform random variables, for all $k \in \N$,
\[
    \E\left[ Q_k(F_1,\ldots,F_k) \right] = Q_k(f,\ldots,f)\,,\qquad  \E\left[ R_k(F_1,\ldots,F_k) \right] = R_k(f,\ldots, f).
\]
It is thus sufficient to prove $Q_k(F_1,\ldots,F_k)\leq R_k(F_1,\ldots,F_k)$ with probability 1.
Let us consider the ordering $U_{(1)} \geq \ldots \geq U_{(k)}$ and let us define the corresponding sets 
$A_{\ell} \triangleq \{ \omega \in \Omega: U_{(\ell)} < f(\omega) \} $.
By Lemma~\ref{lem:dirichlet_products}, for any $\alpha > 0$ and $S\triangleq\sset{\ell_1,\dots,\ell_{\abs{S}}} \subset [k]\triangleq\sset{1,\dots,k}$, $\ell_1<\dots<\ell_{\abs{S}}$, we have
\[
    T(\alpha, S) \triangleq \EEs{X \sim \DP(\alpha \nu_0)}{ \prod_{\ell \in S} X(A_{\ell}) } = \prod_{j =1}^{\abs{S}}  \frac{\alpha \nu_0(A_{\ell_j}) + j - 1}{\alpha + j - 1}.
\]
In particular, if $k \in S$, then $k = \ell_{|S|}$ and
\begin{equation}\label{eq:T_alpha_S_recursion}
    T(\alpha, S) = T(\alpha, S \setminus \{ k \}) \cdot \frac{\alpha \nu_0(A_{k}) + |S|-1}{\alpha + |S| - 1}\,.
\end{equation}
Expanding the products inside the expectation in the definition of $Q_k$ and using the independence between $X$ and $X'$, we have
\[
    Q_k\pa{F_1,\dots,F_k} = \sum_{S \subset [k]} \underbrace{\alpha^{|S|} \beta^{k - |S|} T(\alpha, S) \cdot T(\beta, [k] \setminus S)}_{A(k,S)}.
\]
By \eqref{eq:T_alpha_S_recursion}, we get
\begin{align*}\sum_{S \subset [k]} A(k,S) &= \sum_{S \subset [k] : k\in S} A(k,S) + \sum_{S \subset [k] : k\notin S} A(k,S)
=
\sum_{S \subset [k-1]} \pa{A(k,S\cup\sset{k}) + A(k,S)}
\\&=
\sum_{S \subset [k-1]} \alpha^{|S|} \beta^{k-|S|}\pa{\frac{\alpha}{\beta}T(\alpha,S\cup\sset{k})T(\beta, [k-1]\backslash S) + T(\alpha,S)T(\beta, [k]\backslash S)}
\\&=
\sum_{S \subset [k-1]} A(k-1,S)  \left( \alpha \cdot \frac{\alpha \nu_0(A_{k}) + |S|}{\alpha + |S|} + \beta \cdot \frac{\beta \nu_0(A_{k}) + k  -1 -  |S|}{\beta + k - 1 - |S| } \right).
\end{align*}
Finally, we apply Proposition~\ref{prop:easyconcave} for $z = |S| \in [0, k-1]$ and get 
\[
     \alpha \cdot \frac{\alpha \nu_0(A_{k}) + |S|}{\alpha + |S|} + \beta \cdot  \frac{\beta \nu_0(A_{k}) + k - |S| -1}{\beta + k - |S| - 1}  \leq (\alpha + \beta)  \cdot \frac{(\alpha + \beta) \nu_0(A_k) + k - 1}{\alpha + \beta + k - 1},
\]
so that $\sum_{S \subset [k]} A(k,S)
\leq
(\alpha + \beta) \frac{(\alpha + \beta) \nu_0(A_k) + k - 1}{\alpha + \beta + k - 1}
\sum_{S \subset [k-1]} A(k-1,S)
$. 
By a simple induction on $ k\in \N$, we thus get
$\sum_{S \subset [k]} A(k,S)
\leq (\alpha + \beta)^k \prod_{\ell\in [k]}\frac{(\alpha + \beta) \nu_0(A_{\ell}) + \ell - 1}{\alpha + \beta + \ell - 1} = \pa{\alpha+\beta}^kT(\alpha+\beta,[k]) = R_k\pa{F_1,\dots,F_k}$.
\end{proof}

\begin{lemma}\label{lem:dirichlet_products}For any increasing sequence $A_1\subset\dots\subset A_m,~A_i\in \cB\pa{\Omega}$, $m\in \N^*$, \[\EEs{X \sim \DP(\alpha \nu_0)}{ \prod_{\ell \in [m]} X(A_{\ell}) } = \prod_{\ell\in [m]}\frac{\alpha\nu_0\pa{A_{\ell}} + \ell-1}{\alpha + \ell-1}.\]
\end{lemma}
\begin{proof}
We proceed by induction. For $m=1$, this is simply the expectation formula for the DP. Assume this is true for $m-1$, for any DP and for any measurable increasing sequence of length $m-1$.
Now, let's fix some measurable increasing sequence $A_1\subset\dots\subset A_m\subset \Omega$. Consider any finite measurable partition $\sqcup_{i\in [k]}B_i = A_m$.
Then, by definition of the DP, $(X(B_1),\dots,X(B_k),1-X(A_m))\sim\Dir\pa{\alpha\nu_0(B_1),\dots,\alpha\nu_0(B_k),\alpha(1-\nu_0(A_m))}$.
By the neutrality property of the Dirichlet distribution (see e.g. \cite{james1980new}), it holds
$$X(A_{m}) \indep \pa{\frac{X(B_{1})}{X(A_{m})},\dots,\frac{X(B_{k})}{X(A_{m})}}\sim \Dir\pa{\alpha\nu_0(B_1),\dots,\alpha\nu_0(B_k)}.$$ This is true for all finite measurable partition $\sqcup_{i\in [k]}B_i = A_m$, so $X(A_{m}) \indep\frac{X\pa{\cdot \cap A_m}}{X(A_m)}\sim \DP\pa{\alpha\nu_0\pa{\cdot \cap A_m}}$. From the induction hypothesis on $\frac{X\pa{\cdot\, \cap A_m}}{X(A_m)}$ and the moment formula for the beta distribution, we get
\begin{align*}\EEs{X \sim \DP(\alpha \nu_0)}{ \prod_{\ell \in [m]} X(A_{\ell}) }&=
     \EEs{X \sim \DP(\alpha \nu_0)}{ \prod_{\ell \in [m-1]} \frac{X(A_{\ell})}{X(A_{m})} }\EEs{X \sim \DP(\alpha \nu_0)}{{X(A_{m})}^{m}}\\
     &=
     \prod_{\ell\in [m-1]}\frac{\alpha\nu_0\pa{A_{\ell}} + \ell-1}{\alpha\nu_0(A_m) + \ell-1}
     \cdot
     \prod_{\ell\in [m]} \frac{\alpha\nu_0\pa{A_{m}} +\ell-1}{\alpha +\ell-1}
     \\&=
     \prod_{\ell\in [m]}\frac{\alpha\nu_0\pa{A_{\ell}} + \ell-1}{\alpha + \ell-1}.
\end{align*}
\end{proof}

\begin{proposition}\label{prop:easyconcave} Let $s,t,j>0$ and let $x\in [0,1]$. For $z\in [0,j]$, we define
$$h(z)\triangleq{\frac{\pa{sx+z}s}{s+z}+\frac{\pa{tx+j-z}t}{t+j-z}}.$$
Then, 
$$\max_{z\in[0,j]}{h(z)}=h\pa{\frac{js}{s+t}}=\frac{(s+t)((s+t)x+j)}{s+t+j}.$$
\end{proposition}
\begin{proof}
We have that $h$ is concave on $[0,j]$, as 
$$h'(z)=\frac{s^2(1-x)}{(s+z)^2}-\frac{t^2(1-x)}{(t+j-z)^2}$$
and
$$h''(z)= - \frac{2s^2(1-x)(s+z)}{(s+z)^4} -\frac{2t^2(1-x)(t+j-z)}{(t+j-z)^4} \leq 0.$$
Solving for $h'(z)=0$, we get that the maximizer of $h$ on $[0,j]$ is $z^*\triangleq js/(s+t)$.
Thus,
$$\max_{z\in[0,j]}{h(z)} = h(z^*)= \frac{(s+t)((s+t)x+j)}{s+t+j}.$$
\end{proof}

\section{Application to the Combinatorial Thompson Sampling policy}
\label{sec:cts}
Employing multiple independent DPs as described in our Corollary~\ref{cor:indepDPs} has practical applications in various contexts. One notable example is the stochastic semi-bandit problem  \cite{kveton2015tight}, an extension of the standard multi-armed bandits (MAB) problem. An example of a semi-bandit problem is formulated as follows. We consider a set of $n$ independent Bernoulli distributions (referred to as base arms), each with an unknown mean $\theta_k\in [0,1]$, $k \in [n] \triangleq \sset{1, \dots, n}$. At each round $t\in \N^*$, an agent selects an action $A_t\in \cA$, where $\cA\subset \cP([n])$ is a fixed action space. For each base arm $k$ in action $A_t$, an outcome $X_{k,t}\sim\text{Ber}\pa{\theta_k}$ is drawn independently from the environment and observed as feedback. The agent gains a reward $\sum_{k\in A_t}X_{k,t}$ before moving to the next round. The agent's goal is to minimize the expected regret over $T\in \N^*$ rounds, $R_T\triangleq T\sum_{k\in A^*}\theta_k - \sum_{t\in [T]}\EE{\sum_{k\in A_{t}}\theta_k}$, where $A^*\in\argmax_{A\in \cA}\sum_{k\in A}\theta_k$.

A commonly used policy for this problem is Combinatorial Thompson Sampling (CTS) \cite{wang2018thompson,perrault2020statistical}. In CTS, each base arm $k\in [n]$ is associated with a maintained 
prior distribution $\text{Beta}\pa{1+N_{k,t}\bar\theta_{k,t},1+N_{k,t}(1-\bar\theta_{k,t})}$, where, at the beginning of round $t$,
$N_{k,t}$ (resp. $\bar\theta_{k,t}$) is the number of observations (resp. empirical
mean) of base arm $k$. At round $t$, the agent draws, for each base arm $k$, an independent sample $\theta^+_{k,t}$ from the corresponding prior.
Then, the action to be played is chosen as $A_t\in\argmax_{A\in \cA}\sum_{k\in A}\theta^+_{k,t}$ (we assume that linear optimization is computationally efficient over $\cA$). 

CTS can be compared with two well-known alternative policies: CUCB \cite{chen2016combinatorial,pmlr-v89-perrault19a} and ESCB \cite{combes2015combinatorial}. 
Unlike CTS, which is based on sampling, these policies directly build a confidence region $C_{t}$, for the vector of outcomes, and then play an action $A_t\in\argmax_{A\in \cA}\max_{ (\theta^+_{1,t},\dots,\theta^+_{K,t})\in C_t}\sum_{k\in A} \theta^+_{k,t}$. CUCB is conservative but efficient, using the Cartesian product of the individual outcome confidence intervals, 
whereas ESCB leverages stochastic independence between the base arms but is generally inefficient. 
CTS strikes a good balance by leveraging independence while remaining efficient. These three policies are often qualified as optimistic, which essentially means that the estimates $\theta^+_{k,t}$ are such that the event $\sset{\sum_{k\in A^*}\theta_{k} \leq \sum_{k\in A^*}\theta^+_{k,t}}$ occurs with high probability.

To compare these policies in terms of expected regret, consider a simple semi-bandit instance where $\cA \triangleq \sset{\sset{1,\dots,m},\sset{m+1,\dots,2m},\dots,\sset{n-m+1,\dots,n}}$, with $n,m\in \N$ so that $\abs{\cA}=n/m\in \N$.
We assume that each base arm in an action $j\in [n/m]$ follows an independent Bernoulli distribution of parameter $p_j$ (so $\theta_k=p_{\lceil{k/m}\rceil}$ for $k\in [n]$), and that $p_1=\max_{j\in [n/m]} p_j$. 
This reduces to a MAB problem with $n/m$ actions and with a binomial reward $\text{Bin}(m,p_j)$ for each action $j\in [n/m]$.  First, using a result from \cite{Lai1985asymptotically}, we have the following lower bound on the asymptotic expected regret of any policy:  
\begin{align*}\liminf_{T \to\infty} \frac{R_T}{\log T} \geq \sum_{j=2}^{n/m} \frac{m(p_1-p_j)}{\KL{\text{Bin}(m,p_j)}{\text{Bin}(m,p_1)}} = \sum_{j=2}^{n/m} \frac{p_1-p_j}{\kl{p_j}{p_1}}.\label{cts:lower_bound}\end{align*}
Now, let us examine the upper bounds on \( \limsup_{T \to \infty} R_T/ \log T \), focusing on the specific semi-bandit instance described earlier for the sake of simplicity.
When considering versions of the policies based on a KL confidence region,
CUCB has an upper bound that is $m$ times larger than the lower bound mentioned earlier \cite{chen2016combinatorial}, whereas ESCB's upper bound matches this lower bound \cite{combes2015combinatorial}. Since CTS aims to match the statistical performance of ESCB, a natural question arises: Can CTS achieve the same upper bound as ESCB?
This question can be addressed using Corollary~\ref{cor:indepDPs},  as we will demonstrate next.
 
 Let $\varepsilon>0$. We have
 $R_T= {\sum_{j=2}^{n/m}\sum_{t=1}^T \PP{A_{t}=\text{ action }j}m(p_1 - p_j)},$ so it is sufficient to get $\sum_{t=1}^T \PP{A_{t}=\text{ action }j}\leq\frac{(1+\varepsilon)\log(T)}{m \kl{p_j}{p_1}} + o\pa{\log(T)}$. We can thus set aside the rounds where \( N_{jm,t} \leq \frac{(1+\varepsilon)\log(T)}{m \kl{p_j}{p_1}}\), constituting the leading bound (notice that all the counters $N_{(j-1)m+1,t},\dots,N_{jm,t}$ are equal). In addition, we can focus on the intersection of several high-probability events, which are listed as follows.
 \begin{itemize}
     \item The optimism event: $mp_1\leq \sum_{k\in A^*}\theta^+_{k,t}$.
     \item For all $k\in A_t$, ${\bar\theta_{k,t}}\leq{p_1}$.
     \item $(1+\varepsilon)\sum_{k\in A_t}\kl{\bar\theta_{k,t}}{p_1}\geq \sum_{k\in A_t}\kl{\theta_{k}}{p_1} = m\kl{p_j}{p_1}$ (see \cite{maillard2011finite}).
 \end{itemize}
 In summary, it is sufficient to show that these events are mutually exclusive for the remaining rounds where \( N_{jm,t} > \frac{(1+\varepsilon)\log(T)}{m \kl{p_j}{p_1}} \).
From the optimism event and the policy's definition we have
$mp_1\leq \sum_{k\in A^*}\theta^+_{k,t} \leq \sum_{k\in A_{t}}\theta^+_{k,t}.$ From Corollary~\ref{cor:indepDPs}, we get that this event holds with a conditional probability bounded by $\exp\pa{- N_{jm,t}\sum_{k\in A_t}\kl{\bar\theta_{k,t}}{p_1}}$. Thus, with high probability, $N_{jm}\sum_{k\in A_t}\kl{\bar\theta_{k,t}}{p_1} \leq \log(T)$. This, with \( N_{jm,t} > \frac{(1+\varepsilon)\log(T)}{m \kl{p_j}{p_1}} \), contradicts the last event listed above.

We have demonstrated, using a toy example, how our results can establish the statistical optimality of the CTS policy. These findings could be valuable in more general and practical contexts. Specifically, it would be interesting to show that the expected regret rate of CTS aligns with that of ESCB in problems where ESCB is computationally inefficient.

\section{Conclusion}\label{sec:ccl}

In this paper, we presented a new method for bounding the cumulant generating function (CGF) of Dirichlet Processes (DPs). The proposed non-asymptotic bound achieves asymptotic optimality as \( \alpha \to \infty \). It is expressed as the convex conjugate of \( \alpha \) times the large deviation principle rate function for the DP, represented by the reversed Kullback-Leibler divergence. This approach enables the construction of confidence regions for sums of independent DPs, making it useful for various applications.



\bibliographystyle{amsplain}
\bibliography{ref}

\end{document}